\documentclass[a4paper,11pt, reqno]{amsart}
\usepackage[]{hyperref}
\usepackage{a4wide}
\usepackage{amsthm}
\usepackage{amssymb}
\usepackage{palatino}
\usepackage{enumitem}
\usepackage{relsize}
\usepackage{mathabx}
\usepackage{tcolorbox}
\usepackage{amssymb}
\newtheorem{thm}{Theorem}

\newtheorem{definition}{Definition}

\newtheorem{corollary}{Corollary}[section]
\newtheorem{lemma}{Lemma}

\newtheorem{remark}{Remark}
\makeatletter
\newcommand{\sqf}[1]{\ \mathrm{sqf}\ #1}
\newcommand{\Mod}[1]{\ (\mathrm{mod}\ #1)}
\newcommand*{\rom}[1]{\expandafter\@slowromancap\romannumeral #1@}

\newcommand{\Gal}[1]{\text{Gal}(#1)}
\def\blfootnote{\gdef\@thefnmark{}\@footnotetext}
\makeatother
\def\house#1{\setbox1=\hbox{$\,#1\,$}
\dimen1=\ht1 \advance\dimen1 by 2pt \dimen2=\dp1 \advance\dimen2 by 2pt
\setbox1=\hbox{\vrule height\dimen1 depth\dimen2\box1\vrule}%
\setbox1=\vbox{\hrule\box1}%
\advance\dimen1 by .4pt \ht1=\dimen1
\advance\dimen2 by .4pt \dp1=\dimen2 \box1\relax}

\begin{document}
\title{On the explicit Galois group of $\mathbb{Q}(\sqrt{a_{1}}, \sqrt{a_{2}}, \dots, \sqrt{a_{n}}, \zeta_{d})$ over $\mathbb{Q}$}
\author[C.G. Karthick Babu, Anirban Mukhopadhyay and Sehra Sahu]
{C. G. Karthick Babu$^{(1)}$, Anirban Mukhopadhyay$^{(2)}$ and Sehra Sahu$^{(3)}$}
\date{}
\address{$^{(1)}$ Department of Mathematical Sciences, IISER Berhampur, Govt. ITI,
Engineering School Road, Berhampur-760010, India}
\address{$^{(2)}$Institute of Mathematical Sciences, HBNI,
 C.I.T Campus, Taramani, Chennai-600113, India}
\address{$^{(3)}$ Department of Mathematical and Computational Sciences, National Institute of Technology Karnataka, Surathkal, Srinivasanagar-575025, India}
\email{$^{(1)}$cgkbabu@iiserbpr.ac.in, cgkarthick24@gmail.com}
\email{$^{(2)}$anirban@imsc.res.in}
\email{$^{(3)}$sehrasahu@gmail.com}
 
\begin{abstract}
Let $S= \{ a_{1}, a_{2}, \dots, a_{n} \}$ be a finite set of non-zero integers. In \cite{KBAM21}, Karthick Babu and Anirban Mukhopadhyay calculated the explicit structure of the Galois group of multi-quadratic field $\mathbb{Q}(\sqrt{a_{1}}, \sqrt{a_{2}}, \dots, \sqrt{a_{n}})$ over $\mathbb{Q}$. For a positive integer $d \geq 3$, $\zeta_{d}$ denotes the primitive $d$-th root of unity. In this paper, we calculate the explicit structure of the Galois group of \newline $\mathbb{Q}(\sqrt{a_{1}}, \sqrt{a_{2}}, \dots, \sqrt{a_{n}}, \zeta_{d})$ over $\mathbb{Q}$ in terms of its action on  $\zeta_{d}$ and $\sqrt{a_{i}}$ for $1 \leq i \leq n$.
\end{abstract}

\maketitle
\section{Introduction} 
Let $S= \{ a_{1}, a_{2}, \dots, a_{n} \}$ be a finite set of non-zero integers. In 1968, M. Fried \cite{MFRI68} showed that there are infinitely many primes $p$ for which all the elements of $S$ are quadratic residues. He also provided a necessary and sufficient condition for $a_{i}$'s to be  quadratic non-residues modulo $p$. In 2010, R. Balasubramanian, F. Luca and R. Thangadurai (\cite{BLT10}, Theorem 2.3) calculated the exact density of such primes in Fried's result.

\vspace{1mm} 
A function $\theta : S \rightarrow \{-1, 1\}$ is called \textbf{choice of signs} for $S$. 
\begin{definition}
For a prime $p$, we will say that $S$ has \textbf{residue pattern $\theta$ modulo $p$} if 
\begin{equation*}
\left(\frac{s}{p}\right) = \theta(s), \quad \forall s \in S,
\end{equation*}
where $\left(\frac{\cdot}{p}\right)$ is the Legendre symbol mod $p$.
\end{definition}

Recently in \cite{KBAM21}, the first two authors calculated the exact density of the set of primes for which $S$ has residue pattern $\theta$ modulo $p$. The first two authors also obtained a necessary and sufficient condition for a choice of signs $\theta$ for $S$ to be a residue pattern modulo $p$ for positive density of primes. 
\noindent
 
\vspace{1mm}
\noindent
For any non-zero integer $n=\prod_{i=1}^{m}p_{i}^{k_{i}}$, we define the squarefree part of $n$ by $$\sqf(n)=\prod_{i=1}^{m}p_{i}^{r_{i}}, \ \text{where} \ r_{i} \equiv k_{i} \Mod 2.$$ For any finite subset $T \subset \mathbb{Z}\setminus \{0\}$, by $\sqf(T)$ we mean $\sqf(\prod_{s \in T}s)$. For any nonempty subset $T \subseteq S$, we set
\begin{equation}\label{prod theta}
\theta(T):= \prod_{s \in T} \theta(s).
\end{equation}
Further, we define
\begin{equation*}
S(N, \theta)=  \big\lbrace N < p \leq 2N : p \in \mathrm{P} \ \text{and} \ S \ \text{has residue pattern} \ \theta \Mod p \big\rbrace, 
\end{equation*}
where $\mathrm{P}$ is the set of all primes. 
The first result of this paper is the following:
\begin{thm}\label{residue thm}
Let  $S$ be a finite set of non-zero integers with a choice of signs $\theta$ for $S$. Then for any sufficiently large integer $N \geq 3$, for any integers $1 \leq f \leq d \leq (\log N)^{A}$ such that $(f,d)=1, A>0$, we have
\begin{equation*}
\sum_{\substack{p \in S(N, \theta) \\ p \equiv f (d) }} \log p = \frac{C(S, \theta, d)}{2^{|S|}} \cdot \frac{N}{\varphi(d)}+O\left(\frac{N \log N} {\exp(c(\log N)^{1/2})}\right),
\end{equation*}  
where $c$ is a positive constant which depends only on $A$ and $S$. Here,
\begin{equation}\label{mainterm cons}
C(S, \theta, d)= 
\begin{cases} 
\sum\limits_{\substack{T \subseteq S \\ \sqf(T) \mid d \\ \sqf(T) \equiv 1 (4)}}\left(\frac{\sqf(T)}{f}\right) \theta(T), \ \text{if} \ \ \ 4 \nmid d, \\
\sum\limits_{\substack{T \subseteq S \\ \sqf(T) \mid d \\ \sqf(T) \equiv 1, 3 (4)}}\left(\frac{\sqf(T)}{f}\right) \theta(T), \ \text{if} \ \ \ 4 \mid d, \ 8 \nmid d,\\
\quad \quad \sum\limits_{\substack{T \subseteq S \\ \sqf(T) \mid d}}\left(\frac{\sqf(T)}{f}\right)\theta(T), \ \text{if} \ \ \ 8 \mid d.
\end{cases}
\end{equation}  
\end{thm}

In Section \ref{pf of thm 1}, we will discuss the proof of Theorem \ref{residue thm}. 
Then in Section \ref{cor of thm 1}, we prove some combinatorial lemmas, and as an application of those lemmas, we discuss the positivity of $C(S, \theta, d)$ in Lemma \ref{mainterm lem}. Also, we discuss some Corollaries of Theorem \ref{residue thm} in Section \ref{cor of thm 1}. Especially in the Corollary \ref{inf primes cor}, we discuss a necessary and sufficient condition for a choice of signs $\theta$ for $S$ to be a residue pattern modulo $p$ for infinitely many primes $p$ of the form $p \equiv f \Mod d$.

\vspace{1mm}
In Section \ref{counting cos}, we count the number of choice of signs $\theta$ for $S$ such that $S$ has residue pattern $\theta$ modulo $p$ for infinitely many primes $p$ of the form $p \equiv f \Mod d$.
\vspace{2mm}

It is well-known that the degree of the multi-quadratic field $\mathbb{Q}(\sqrt{a_{1}}, \sqrt{a_{2}}, \dots, \sqrt{a_{n}})$
over $\mathbb{Q}$ is $2^{t}$ for some integer $0 \leq t \leq n$ depending on the algebraic cancellations among the $\sqrt{a_{i}}$'s. The arithmetic of multi-quadratic number fields plays a crucial role in the theory of elliptic curves (see \cite{ACZ95}, \cite{MLML85}). In \cite{BLT10}, R. Balasubramanian, F. Luca and R. Thangadurai gave an exact formula for the degree of the multi-quadratic field $\mathbb{Q}(\sqrt{a_{1}}, \sqrt{a_{2}}, \dots, \sqrt{a_{n}})$ over $\mathbb{Q}$. Recently in \cite{KBAM21}, the first two authors calculated the explicit structure of the Galois group of multi-quadratic field $\mathbb{Q}(\sqrt{a_{1}}, \sqrt{a_{2}}, \dots, \sqrt{a_{n}})$ over $\mathbb{Q}$. For a positive integer $d \geq 3$, $\zeta_{d}$ denotes the primitive $d$-th root of unity. In Section \ref{isomorphism sec}, we calculate the explicit structure of the Galois group of $\mathbb{Q}(\sqrt{a_{1}}, \sqrt{a_{2}}, \dots, \sqrt{a_{n}}, \zeta_{d})$ over $\mathbb{Q}$ in terms of its action on  $\zeta_{d}$ and $\sqrt{a_{i}}$ for $1 \leq i \leq n.$ Further, in Corollary \ref{multi cyclo extn}, we discuss the explicit structure of the Galois group of $\mathbb{Q}(\sqrt{a_{1}}, \dots, \sqrt{a_{n}}, \zeta_{d_{1}}, \dots, \zeta_{d_{k}})$ over $\mathbb{Q}$, where $\zeta_{d_{i}}$ denotes the primitive $d_{i}$-th roots of unity, for $1 \leq i \leq k$. In the last section, we distinguish the algebraic cancellations coming from the multi quadratic part and the cyclotomic part.

\section{Notation}
We recall that the notation $f = O(g)$ is equivalent to the assertion that the inequality $|f| \leq c g $ holds for some constant $c > 0$. As usual, $p$ denotes a prime and $|A|$ denotes the cardinality of the set $A.$ $\varphi(n)$ denotes the Euler phi function. For any set $B,$ we write $(p, B)=1,$ to mean $(p, b)=1$ for each $b \in B.$ $A \triangle B$ denotes the symmetric difference of sets $A$ and $B$. $A \setminus B$ denotes set difference of sets $A$ and $B$.
 
\section{Proof of Theorem \ref{residue thm}}\label{pf of thm 1}
\subsection{Arithmetic Lemmas}
Here we list several lemmas required for the proof of Theorem \ref{residue thm}.
\begin{lemma}\label{char lemma}
Let $\chi_{1}$ and $\chi_{2}$ be Dirichlet characters modulo $d_{1}$ and $d_{2}$ respectively such that $(d_{1}, d_{2})=1$. We define $\chi= \chi_{1} \chi_{2}$ as a character modulo $d$, where $d=d_{1}d_{2}.$ Then $\chi$ is the principal character modulo $d$ if and only if $\chi_{1}, \chi_{2}$ are the principal characters modulo $d_{1}, d_{2}$ respectively.
\end{lemma}
\begin{proof}
If $\chi_{1}$ and $\chi_{2}$ are the principal characters modulo  $d_{1}$ and $d_{2}$ respectively. Then it is easy to see that $\chi$ is the principal character modulo $d$.  

Conversely, suppose $\chi= \chi_{1}. \chi_{2}$ is the principal character modulo $d.$ Then by definition, we have $\chi_{2}(n)=\overline{\chi_{1}}(n)$ if $(n, d_{1}d_{2})=1.$ Suppose $1 \neq r_{2}$ be a reduced residue class modulo $d_{2}$ such that $(r_{2},d_{2})=1$. Since $(d_{1}, d_{2})=1$ there is an $m \in \mathbb{Z}^{+}$ such that  $r_{2}+md_{2} \equiv 1 \mod d_{1}.$ Therefore, we have $1=\overline{\chi_{1}}(1)= \chi_{2}(r_{2}+md_{2})=\chi_{2}(r_{2})$ for any $r_{2} < d_{2}$ and $(r_{2},d_{2})=1.$ Thus $\chi_{2}$ is the principal character modulo $d_{2}.$ Similarly, we can show that $\chi_{1}$ is the principal character modulo $d_{1}.$  
\end{proof}
 
\begin{remark}
Above lemma need not be true if $(d_{1}, d_{2})>2.$ For example, consider $\chi_{6}$ and $\chi_{9}$, the non-principal Dirichlet characters of modulus $6$ and modulus $9$ respectively, given by the character tables below
\begin{center}
\begin{tabular}{ |c|c|c|c|c|c|c| } 
 \hline
 n & 0 & 1 & 2 & 3 & 4 & 5 \\ \hline
 $\chi_{6}(n)$ & 0& 1 & 0 & 0 & 0 & -1\\  
 \hline
\end{tabular}
\end{center}

\begin{center}
\begin{tabular}{ |c|c|c|c|c|c|c|c|c|c| } 
 \hline
 n & 0 & 1 & 2 & 3 & 4 & 5& 6 & 7 & 8 \\ \hline
 $\chi_{9}(n)$ & 0 & 1 & -1 & 0 & 1 & -1 & 0 & 1 & -1\\  
 \hline
\end{tabular}
\end{center}
But the character $\chi= \chi_{6}\chi_{9}$ is principal modulo 18. Here we note that both the characters $\chi_{6}$ and $\chi_{9}$ are induced by a non-principal character modulo $(d_{1},d_{2})=3$.

Furthermore, suppose $(d_{1},d_{2})> 1,$ and $\chi_{1}\chi_{2}$ is the principal character modulo $[d_{1},d_{2}]$, then both $\chi_{1} \ \text{and} \ \chi_{2}$ are induced by a character modulo $d'$ which divides $(d_{1}, d_{2}).$ To prove that, it is enough to show 
$$\chi_{2}(r)=1, \ \ \ \text{for any} \ \ r \equiv 1 \Mod{(d_{1}, d_{2})}, \ \text{and}  \ \ (r,d_{2})=1.$$ Since $r \equiv 1 \mod (d_{1}, d_{2})$ we have $r=1+ l (d_{1},d_{2})$ for some $l \in \mathbb{Z}^{+}$. Also, it follows from the linear form of gcd that there is an $m \in \mathbb{Z}^{+}$ such that $1+l(d_{1}, d_{2})+m d_{2} \equiv 1 \mod d_{1},$ which implies
\begin{equation*}
\chi_{2}(r)=\chi_{2}(1+l(d_{1}, d_{2}))= \chi_{2}(1+l(d_{1}, d_{2})+md_{2})=\overline{\chi_{1}}(1)=1.  
\end{equation*}
Similarly, $\chi_{1}$ is also induced by a character modulo $d'$ which divides  $(d_{1},d_{2}).$ But Lemma \ref{char lemma} holds if we assume $(d_{1},d_{2})=2$, since there is no non-principal character modulo 2.
\end{remark}
\vspace{2mm}
\begin{corollary}\label{char cor}
Let $\chi_{1}$ and $\chi_{2}$ be two characters modulo $d_{1}$ and $d_{2}$ respectively such that $d_{1} < d_{2}$ and $d_{1} \nmid d_{2}$. If either $\chi_{1}$ or $\chi_{2}$ is primitive, then $\chi_{1}\chi_{2}$ cannot be the principal character modulo $[d_{1},d_{2}].$
\end{corollary}

\begin{remark}[Ch.5, \cite{DAVBOOK}] 
For every fundamental discriminant $m,$ the Kronecker symbol $\chi_{m} = \left(\frac{m}{\cdot}\right)$ is a primitive quadratic character of conductor $|m|.$ Conversely, given any primitive quadratic character $\chi,$ there exists a unique fundamental discriminant $m$ such that $\chi= \chi_{m}.$ When $m$ is not a fundamental discriminant $\chi_{m}$ might be primitive character (e.g. $\chi_{2}$), nonprimitive character (e.g. $\chi_{4}$) or even not a character at all (e.g. $\chi_{3}$). Therefore, in the following lemma we use the law of reciprocity to transform Kronecker symbol $\left(\frac{m}{\cdot}\right)$  into Jacobi symbol $\left(\frac{\cdot}{m}\right)$, which is Dirichlet character modulo $m$.  
\end{remark}

\begin{lemma}\label{charsum lemma}
Let $N \geq 3$ be a positive integer and $A >0,\ B> 0$ be real numbers. Then there exists a positive real $c=c(A,B)$ such that for positive integers $f, d$ satisfying $1 \leq f \leq d \leq \log ^{A}N$ and $(f, d)=1$,
\begin{equation*}
\sum_{\substack{ N < p \leq 2N \\ p \equiv f (d)}}\left(\frac{s}{p}\right) \ll N \exp(-c(\log N)^{1/2}) 
\end{equation*} 
holds for any squarefree integer $s \leq \log^{B}N$ such that  $s\nmid d$.
\end{lemma}
\begin{proof}
We first assume $s$ is odd. Then we have the following two cases. 

\noindent
\textbf{Case I :} We assume $s \equiv 1 \Mod 4$. Then by the law of quadratic reciprocity, we have
\begin{align*}
\sum_{\substack{ N < p \leq 2N \\ p \equiv f (d)}}\left(\frac{s}{p}\right) = \sum_{\substack{ N < p \leq 2N \\ p \equiv f (d)}}\left(\frac{p}{s}\right) =\frac{1}{\varphi(d)} \sum_{\psi \Mod d } \overline{\psi}(f)\sum_{N<p\leq 2N}\left(\frac{p}{s}\right)\psi(p),
\end{align*}
where the first sum runs over the set of all Dirichlet characters modulo $d.$ Since $s$ is a squarefree odd integer, the Jacobi symbol $\left(\frac{.}{s}\right)$ is a primitive character modulo $s$. By Corollary \ref{char cor}, $\left(\frac{.}{s}\right)\psi$ is a non-principal character modulo $[d,s]$ as $s \nmid d$. Then by Siegel's Theorem (Ch. 22, \cite{DAVBOOK}) we obtain the required estimate.

\vspace{1mm}

Other cases use Reciprocity and Siegel's Theorem similarly.
\end{proof}

The next lemma treats the same character sum when $s$ divides $d$, and as expected, we get the main term.  
\begin{lemma}\label{charsum at d lemma}
Let $N \geq 3$ be an integer and $A>0$ be a real number. Then for a squarefree integer $s$ such that $s \mid d$, there exists a positive real $c=c(A)$ such that for positive integers $f, d$ satisfying $1 \leq f \leq d \leq \log ^{A}N$ and $(f, d)=1$,   we have
\begin{equation*}
\sum_{\substack{ N < p \leq 2N \\ p \equiv f (d)}}\left(\frac{s}{p}\right)=
\begin{cases} 
\left(\frac{s}{f}\right) \sum_{\substack{ N < p \leq 2N \\ p \equiv f (d)}} 1 \ \ \ \text{if} \ \ \ s \equiv 1 \Mod 4, \\
\left(\frac{s}{f}\right) \sum_{\substack{ N < p \leq 2N \\ p \equiv f (d)}} 1 \ \ \ \text{if} \ \ \ s \equiv 3 \Mod 4  \ \text{and} \ 4 \mid d,\\
\left(\frac{s}{f}\right) \sum_{\substack{ N < p \leq 2N \\ p \equiv f (d)}} 1 \ \ \ \text{if} \ \ \ s \equiv 2 \Mod 4  \ \text{and} \ 8 \mid d,\\
O(N \exp(-c(\log N)^{1/2})), \quad \text{otherwise}. 
\end{cases}
\end{equation*} 
\end{lemma}
\begin{proof} 
We assume $s \equiv 1 \Mod 4$. Since $s \mid d$ and $p \equiv f \Mod d$ implies $p \equiv f \Mod s$, by the law of quadratic reciprocity, we have
\begin{align*}
\sum_{\substack{ N < p \leq 2N \\ p \equiv f (d)}}\left(\frac{s}{p}\right) = \sum_{\substack{ N < p \leq 2N \\ p \equiv f (d)}}\left(\frac{p}{s}\right) =\left(\frac{f}{s}\right) \sum_{\substack{ N < p \leq 2N \\ p \equiv f (d)}} 1=\left(\frac{s}{f}\right) \sum_{\substack{ N < p \leq 2N \\ p \equiv f (d)}} 1.
\end{align*}

\noindent
Now we assume $s \equiv 3 \Mod 4$. Suppose $4 \mid d$ then either all primes of the form $p \equiv f \Mod d$ are $1 \Mod 4$ or all are $3 \Mod 4.$ Since $s \mid d$ and $p \equiv f \Mod d$ implies $p \equiv f \Mod s$, by the law of quadratic reciprocity, we have
\begin{align*}
\sum_{\substack{ N < p \leq 2N \\ p \equiv f (d)}}\left(\frac{s}{p}\right) =\sum_{\substack{ N < p \leq 2N \\ p \equiv f (d)}}(-1)^{\frac{p-1}{2}}\left(\frac{p}{s}\right)= \left(\frac{f}{s}\right) \sum_{\substack{ N < p \leq 2N \\ p \equiv f (d)}} (-1)^{\frac{f-1}{2}}=\left(\frac{s}{f}\right) \sum_{\substack{ N < p \leq 2N \\ p \equiv f (d)}} 1.
\end{align*}
Suppose $4 \nmid d$, then by using similar arguments as in Lemma \ref{charsum lemma} we obtain the required estimate. For the case $s \equiv 2 \Mod 4$, the proof again uses Reciprocity and Siegel's Theorem as above. Hence we omit the proof. 
\end{proof}

\vspace{2mm}
\noindent
\textbf{Proof of Theorem \ref{residue thm}: } By using the definition of the Legendre symbol, we can write 
\begin{equation*}
\sum_{\substack{p \in S(N, \theta) \\ p \equiv f (d) }} \log p = \frac{1}{2^{|S|}}\sum_{\substack{N < p \leq 2N \\ p \equiv f (d) \\ (p, S)=1}} \log p \prod_{s \in S} \left(1+\theta(s)\left(\frac{s}{p}\right)\right).
\end{equation*}
The condition $(p, S)=1$ can be omitted for sufficiently large $N$. For instance we can take $(\log N)^{B} > \prod_{s \in S} s$ for some $B > 0$. We write
\begin{equation}\label{prod-symb eq}
\prod_{s \in S}\left(1+\theta(s)\left(\frac{s}{p}\right)\right)= 1+ \sum_{\phi\neq T \subseteq S} \prod_{s \in T}\theta(s)\left(\frac{s}{p}\right)
\end{equation}
Since the empty product is $1$ by convention, using \eqref{prod theta} we rewrite \eqref{prod-symb eq} as
\begin{align*}
\prod_{s \in S}\left(1+\theta(s)\left(\frac{s}{p}\right)\right)=\sum_{T \subseteq S}\theta(T)\prod_{s \in T}\left(\frac{s}{p}\right).
\end{align*}
Therefore,
\begin{align*}
\frac{1}{2^{|S|}}\sum_{\substack{N < p \leq 2N \\ p \equiv f (d)}} \log p \prod_{s \in S} \left(1+\theta(s)\left(\frac{s}{p}\right)\right)=\frac{1}{2^{|S|}} \sum_{ T \subseteq S}\theta(T)\sum_{\substack{ N < p \leq 2N \\ p \equiv f (d)}}\prod_{s \in T} \left(\frac{s}{p}\right)\log p.
\end{align*}

Recall that we assumed $d$ and $\prod_{s \in S} s$ are both bounded by powers of $\log N$. Suppose $4 \nmid d$, then by using Lemma \ref{charsum lemma} along with Lemma \ref{charsum at d lemma} and partial summation formula, we obtain

\begin{align*}
\sum_{\substack{p \in S(N, \theta) \\ p \equiv f (d)}} \log p = \frac{1}{2^{|S|}}\sum\limits_{\substack{T \subseteq S \\ \sqf(T) \mid d \\ \sqf(T) \equiv 1 (4)}}\left(\frac{\sqf(T)}{f}\right) \theta(T)\sum_{\substack{ N < p \leq 2N \\ p \equiv f (d)}}\log p +O\left(\frac{N \log N} {\exp(c(\log N)^{1/2})}\right).
\end{align*}
Suppose $4 \mid d$ and $8 \nmid d$, then by using Lemma \ref{charsum lemma} along with Lemma \ref{charsum at d lemma} and partial summation formula, we obtain
\begin{align*}
\sum_{\substack{p \in S(N, \theta) \\ p \equiv f (d)}} \log p = \frac{1}{2^{|S|}}\sum\limits_{\substack{T \subseteq S \\ \sqf(T) \mid d \\ \sqf(T) \equiv 1,3  (4)}}\left(\frac{\sqf(T)}{f}\right) \theta(T)\sum_{\substack{ N < p \leq 2N \\ p \equiv f (d)}}\log p +O\left(\frac{N \log N} {\exp(c(\log N)^{1/2})}\right).
\end{align*}
Suppose $8 \mid d$, then by using Lemma \ref{charsum lemma} along with Lemma \ref{charsum at d lemma} and partial summation formula, we obtain
\begin{align*}
\sum_{\substack{p \in S(N, \theta) \\ p \equiv f (d)}} \log p = \frac{1}{2^{|S|}}\sum\limits_{\substack{T \subseteq S \\ \sqf(T) \mid d}}\left(\frac{\sqf(T)}{f}\right) \theta(T)\sum_{\substack{ N < p \leq 2N \\ p \equiv f (d)}}\log p +O\left(\frac{N \log N} {\exp(c(\log N)^{1/2})}\right).
\end{align*}
Since $d \leq (\log N)^{A}$ for $A>0$, by using Siegel-Walfisz Theorem we obtain,
\begin{align*}
\sum_{\substack{p \in S(N, \theta) \\ p \equiv f (d) }} \log p= \frac{C(S, \theta, d)}{2^{|S|}}\cdot\frac{N}{\varphi(d)}+O\left(\frac{N \log N} {\exp(c(\log N)^{1/2})}\right)\text{,}
\end{align*}
where $C(S, \theta, d)$ is defined as in \eqref{mainterm cons}. This completes the proof of Theorem \ref{residue thm}.

\section{Some Combinatorial Lemmas and Corollaries of Theorem \ref{residue thm}}\label{cor of thm 1}
Let $S$ be a finite set of non-zero integers and $\mathbf{P}(S)$ denote the set of all subsets of $S$. It is well-known that $(\mathbf{P}(S), \triangle, \cap)$ is a commutative ring (Boolean ring). In fact, $(\mathbf{P}(S), \triangle)$ is an abelian group with the identity element $\phi$ which is isomorphic to $\mathbb{F}_{2} \times \dots \times \mathbb{F}_{2}$ ($|S|$ times), where $\phi$ denotes the empty subset.  
For any given positive integer $d$, we define the following subsets of $\mathbf{P}(S)$.
\begin{align}\label{d-subsets}
\mathcal{D}_{0}=& \left\{T \in \mathbf{P}(S): \sqf(T) \mid d\right\}\nonumber\\
\mathcal{D}_{1}=& \left\{T \in \mathbf{P}(S): \sqf(T) \mid d \ \text{and} \ \sqf(T) \equiv 1 \ \text{or} \ 3 \Mod 4 \right\}\nonumber\\
\mathcal{D}_{2}=& \left\{T \in \mathbf{P}(S): \sqf(T) \mid d \ \text{and} \ \sqf(T) \equiv 1 \Mod 4 \right\}.
\end{align}
Note that $\mathcal{D}_{2} \subseteq \mathcal{D}_{1} \subseteq \mathcal{D}_{0}$. In the following lemma, we prove that $\mathcal{D}_{0}, \mathcal{D}_{1}$ and $\mathcal{D}_{2}$ are subgroups of $\mathbf{P}(S)$.
\begin{lemma}\label{d-subgps lem}
$\mathcal{D}_{0}, \mathcal{D}_{1}$ and $\mathcal{D}_{2}$ are subgroups of $\mathbf{P}(S)$. Moreover, if there exists $T_{1} \in \mathbf{P}(S)$ such that $\sqf(T_{1}) \equiv 2 \Mod 4$ then $\mathcal{D}_{1}$ is an index 2 subgroup of $\mathcal{D}_{0}$. Also, if there exists $T_{1} \in \mathbf{P}(S)$ such that $\sqf(T_{1}) \equiv 3 \Mod 4$ then $\mathcal{D}_{2}$ is an index 2 subgroup of $\mathcal{D}_{1}$.
\end{lemma} 
\begin{proof}
For any $S_{1}, S_{2} \in \mathbf{P}(S)$, we observe that
\begin{equation}\label{sqf of sd of a coset}
\sqf(S_{1} \triangle S_{2})=\sqf (\sqf(S_{1}) \cdot \sqf(S_{2})).
\end{equation}
Let $S_{1}, S_{2} \in \mathcal{D}_{0}$. Equivalently, $\sqf(S_{1}) \mid d$ and $\sqf(S_{2}) \mid d$. Then by \eqref{sqf of sd of a coset} it follows that $\sqf(S_{1} \triangle S_{2}) \mid d$. Thus $\mathcal{D}_{0}$ is a subgroup of $\mathbf{P}(S)$. Now we define a map $\tau_{1}: \mathcal{D}_{0} \rightarrow \{1,-1\}$ by 
\begin{equation*}
\tau_{1}(T)=\begin{cases}
1 \quad \ \ \ \text{if} \ \sqf(T) \ \text{is odd}\\
-1 \quad  \ \text{if} \ \sqf(T)  \ \text{is even}.
\end{cases}
\end{equation*}
For any $T_{1}, T_{2} \in \mathcal{D}_{0}$, using \eqref{sqf of sd of a coset} we obtain  $\tau_{1}(T_{1} \triangle T_{2})=1$ if either both $\sqf(T_{1})$ and $ \sqf (T_{2})$ are even or both are odd and $\tau_{1}(T_{1} \triangle T_{2})=-1$ if one of $\sqf(T_{1})$ and $ \sqf (T_{2})$ is even and other one is odd. Thus $\tau_{1}(T_{1} \triangle T_{2})=\tau_{1}(T_{1})\cdot\tau_{1}(T_{2})$ for every $T_{1}, T_{2} \in \mathcal{D}_{0}$. Therefore, $\tau_{1}$ is a homomorphism with the kernel $\mathcal{D}_{1}$. Suppose there exists $T_{1} \in \mathbf{P}(S)$ such that $\sqf(T_{1}) \equiv 2 \Mod 4$, then $\tau_{1}$ is a surjective homomorphism. Thus $\mathcal{D}_{1}$ is an index $2$ subgroup of $\mathcal{D}_{0}$.

Define a map $\tau_{2}: \mathcal{D}_{1} \rightarrow (\mathbb{Z}_{4}^{*}, \odot_{4})$ by $\tau_{2}(T)= \sqf (T) \Mod 4.$ For any $T_{1}, T_{2} \in \mathcal{D}_{1}$, by \eqref{sqf of sd of a coset} we have  
\begin{align*}
\tau_{2}(T_{1} \triangle T_{2})=&\sqf (\sqf(T_{1}) \cdot \sqf(T_{2})) \Mod 4= \sqf(T_{1}) \sqf(T_{2}) \bigg(\prod\limits_{\substack{p \mid \sqf(T_{1})  \\p \mid\sqf(T_{2})}}p^{2} \bigg)^{-1} \Mod 4\\
=& \sqf(T_{1})\Mod 4 \odot_{4} \sqf(T_{2}) \Mod 4= \tau_{2}(T_{1}) \odot_{4} \tau_{2}(T_{2}).
\end{align*}
Therefore, $\tau_{2}$ is a homomorphism with the kernel $\mathcal{D}_{2}$. Suppose there exists $T_{1} \in \mathbf{P}(S)$ such that $\sqf(T_{1}) \equiv 3 \Mod 4$, then $\tau_{2}$ is a surjective homomorphism. Thus $\mathcal{D}_{2}$ is an index $2$ subgroup of $\mathcal{D}_{1}$.
\end{proof}

\begin{lemma}\label{homo lem}
Let $S$ be a finite set of non-zero integers with a choice  of signs $\theta$ for $S$. Let $H$ be a subgroup of $\mathbf{P}(S)$ and $f$ be a squarefree positive integer such that $(f, \sqf(T))=1$ for every $T \in H$. Define 
$$\mu_{f}^{\theta}: H \rightarrow \{1,-1\} \ \ \ \text{by} \ \ \ \ \mu_{f}^{\theta}(T)= \theta(T) \cdot \left(\frac{\sqf(T)}{f}\right).$$ Then $\mu_{f}^{\theta}$ is a homomorphism. Suppose there exists $T \in H$ such that $\theta(T) \neq \left(\frac{\sqf(T)}{f}\right)$ then $\mu_{f}^{\theta}$ is a surjective homomorphism.
\end{lemma}
\begin{proof}
For any $T_{1}, T_{2} \in H$, we observe that
\begin{equation*}\label{theta mult}
\theta(T_{1} \triangle T_{2})= \prod\limits_{s \in T_{1} \triangle T_{2}}\theta(s)= \left(\prod\limits_{s \in T_{1}}\theta(s)\right) \left(\prod\limits_{s \in T_{2}}\theta(s)\right) \left(\prod\limits_{s \in T_{1} \cap T_{2}}(\theta(s))^{2}\right)^{-1}=\theta(T_{1}) \theta(T_{2}).
\end{equation*}
Also, by completely multiplicative property of Jacobi symbol, we have
\begin{equation*}
\left(\frac{\sqf(T_{1})}{f}\right)\cdot\left(\frac{\sqf(T_{2})}{f}\right)=\left(\frac{\sqf(T_{1})\sqf(T_{2})}{f}\right)=\left(\frac{\sqf(\sqf(T_{1})\sqf(T_{2}))}{f}\right)=\left(\frac{\sqf(T_{1}\triangle T_{2})}{f}\right).
\end{equation*}
Thus, $\mu_{f}^{\theta}(T_{1}\triangle T_{2})=\mu_{f}^{\theta}(T_{1})\cdot\mu_{f}^{\theta}(T_{2})$ holds for every $T_{1}, T_{2} \in H$. Hence $\mu_{f}^{\theta}$ is a homomorphism.
\end{proof}

\vskip 2mm
In the following lemma, we show $C(S, \theta, d) \geq 0$ and discuss the necessary and sufficient condition for $C(S, \theta, d) > 0$. 
\begin{lemma}\label{mainterm lem} 
Let $S$, $\theta$, $f$ and $d$ be stated as in Theorem \ref{residue thm}. Let $\mathcal{D}_{0}, \mathcal{D}_{1}$ and $\mathcal{D}_{2}$ be subsets of $\mathbf{P}(S)$ as defined in $\eqref{d-subsets}$. Then the constant $C(S, \theta, d)$ appearing in the main term of Theorem \ref{residue thm} is always non-negative. Moreover, 
\begin{enumerate}[label=(\roman*)] 
\item if $4 \nmid d$, then $C(S, \theta, d) >0$ holds if and only if $\theta(T)=\left(\frac{\sqf(T)}{f}\right)$ for every $T \in \mathcal{D}_{2}$,
\item if $4 \mid d$ and $8 \nmid d$, then $C(S, \theta, d) >0$ holds if and only if $\theta(T) =\left(\frac{\sqf(T)}{f}\right)$ for every $T \in \mathcal{D}_{1}$,
\item if $8 \mid d$, then $C(S, \theta, d) >0$ holds if and only if $\theta(T)=\left(\frac{\sqf(T)}{f}\right)$ for every $T \in \mathcal{D}_{0}$.
\end{enumerate}
\end{lemma}
\begin{proof}
By using \eqref{d-subsets}, we rewrite \eqref{mainterm cons} as
\begin{equation}\label{mainterm cons in D-1}
C(S, \theta, d)=\begin{cases}
\sum\limits_{T \in \mathcal{D}_{2}} \left(\frac{\sqf(T)}{f}\right)\theta(T) &\ \text{if} \ 4 \nmid d,\\
\sum\limits_{T \in \mathcal{D}_{1}} \left(\frac{\sqf(T)}{f}\right)\theta(T) &\ \text{if} \ 4 \mid d, 8 \nmid d,\\
\sum\limits_{T \in \mathcal{D}_{0}} \left(\frac{\sqf(T)}{f}\right)\theta(T) &\ \text{if} \ 8 \mid d.\\
\end{cases}
\end{equation}
From Lemma \ref{d-subgps lem}, we know that $\mathcal{D}_{0}, \mathcal{D}_{1}$ and $\mathcal{D}_{2}$ are subgroups of $\mathbf{P}(S)$. Since $(f,d)=1$ and $\sqf(T) \mid d$ we have $(f, \sqf(T))=1$ for every $T \in \mathcal{D}_{i}$, for $i=0,1,2$. Hence, it follows from Lemma \ref{homo lem} that $\mu_{f}^{\theta}: \mathcal{D}_{i} \rightarrow \{1,-1\}$ is a homomorphism for $1=0,1,2$. Thus, $C(S, \theta, d) \geq 0$ always holds. 

\noindent
In each case the corresponding sum in \eqref{mainterm cons in D-1} is positive if and only if $\theta(T) =\left(\frac{\sqf(T)}{f}\right)$ for every $T \in \mathcal{D}_{i}$, for $i=0,1,2$. Thus, in each case $C(S, \theta, d) >0$ holds if and only if $\theta(T) =\left(\frac{\sqf(T)}{f}\right)$ for every $T \in \mathcal{D}_{i}$, for $i=0,1,2$. This completes the proof of Lemma \ref{mainterm lem}.
\end{proof}

\subsection{Some Corollaries of Theorem \ref{residue thm}}
\begin{remark}\label{rmk suf}
Given $S \subseteq \mathbb{Z} \setminus \{0 \}$ with a choice of signs $\theta$ for $S$ and positive integers $f, d$  with $1\leq f \leq d$ and $(f,d)=1$. Suppose $4 \nmid d$ and there is a subset $T \in \mathcal{D}_{2}$ such that $\theta(T) \neq \left(\frac{\sqf(T)}{f}\right)$, then there is no prime $p$ of the form $p \equiv f \Mod d$ such that $S$ has residue pattern $\theta$ modulo $p$. It is not surprising, because if such a prime $p \equiv f \Mod d$ exists then from the proof of Lemma \ref{charsum at d lemma} it is easy to see that 
\begin{align*}
-1=\left(\frac{\sqf(T)}{f}\right)\theta(T)= &\left(\frac{\sqf(T)}{f}\right)\prod_{s \in T}\theta(s)=\left(\frac{\sqf(T)}{f}\right)\prod_{s \in T}\left(\frac{s}{p}\right)\\
=&\left(\frac{\sqf(T)}{f}\right)\left(\frac{\sqf(T)}{p}\right)=\left(\frac{\sqf(T)}{f}\right)^{2}=1,
\end{align*}
which is absurd. 

For the other two cases, one can similarly show that there exists a prime $p$ of the form $p \equiv f \Mod d$ such that $S$ has residue pattern $\theta$ modulo $p$ only if the sufficient condition stated in Lemma \ref{mainterm lem} holds. Thus, by using Theorem \ref{residue thm} and Lemma \ref{mainterm lem}, we have the following corollary.  
\end{remark}

\begin{corollary}\label{inf primes cor}
Let $S$ be a finite set of non-zero integers with a choice of signs $\theta$ for $S$. Let $f, d$ be positive integers with $1\leq f \leq d$ and $(f,d)=1$. 
\begin{description}
\item [Case 1] $4 \nmid d$ \\
There exist infinitely many primes $p$ of the form $p \equiv f \Mod d$ such that $S$ has residue pattern $\theta$ modulo $p$ if and only if $\theta(T)=\left(\frac{\sqf(T)}{f}\right)$ for every $T \in \mathcal{D}_{2}$. In this case, the asymptotic density of primes of the form $p \equiv f \Mod d$ for which $S$ has residue pattern $\theta$ modulo $p$ is  $\frac{|\mathcal{D}_{2}|}{\varphi(d)2^{|S|}}$.  
\item [Case 2] $4 \mid d$ and $8 \nmid d$ \\
There exist infinitely many primes $p$ of the form $p \equiv f \Mod d$ such that $S$ has residue pattern $\theta$ modulo $p$ if and only if $\theta(T)=\left(\frac{\sqf(T)}{f}\right)$ for every $T \in \mathcal{D}_{1}$. In this case, the asymptotic density of primes of the form $p \equiv f \Mod d$ for which $S$ has residue pattern $\theta$ modulo $p$ is  $\frac{|\mathcal{D}_{1}|}{\varphi(d)2^{|S|}}$.  
\item [Case 3] $8 \mid d$ \\
There exist infinitely many primes $p$ of the form $p \equiv f \Mod d$ such that $S$ has residue pattern $\theta$ modulo $p$ if and only if $\theta(T)=\left(\frac{\sqf(T)}{f}\right)$ for every $T \in \mathcal{D}_{0}$. In this case, the asymptotic density of primes of the form $p \equiv f \Mod d$ for which $S$ has residue pattern $\theta$ modulo $p$ is  $\frac{|\mathcal{D}_{0}|}{\varphi(d)2^{|S|}}$.   
\end{description}
\end{corollary} 

\begin{corollary}\label{allpattern cor}
Let $f, d$ be positive integers with $1\leq f \leq d$ and $(f,d)=1$.
\begin{description}
\item [Case 1] $4 \nmid d$ \\
A nonempty finite set $S \subset \mathbb{Z} \setminus \{0\}$ supports all residue patterns for infinitely many primes $p$ of the form $p \equiv f \Mod d$ if and only if $\mathcal{D}_{2}=\{\phi\}$. 
\item [Case 2] $4 \mid d$ and $8 \nmid d$ \\
A nonempty finite set $S \subset \mathbb{Z} \setminus \{0\}$ supports all residue patterns for infinitely many primes $p$ of the form $p \equiv f \Mod d$ if and only if $\mathcal{D}_{1}=\{\phi\}$.
\item [Case 3] $8 \mid d$ \\
A nonempty finite set $S \subset \mathbb{Z} \setminus \{0\}$ supports all residue patterns for infinitely many primes $p$ of the form $p \equiv f \Mod d$ if and only if $\mathcal{D}_{0}=\{\phi\}$. 
\end{description}  
In each case, for every choice of signs $\theta : S \rightarrow \{-1, 1\}$ the density of the set $S(N, \theta)$ is $\frac{1}{\varphi(d)2^{|S|}}$.
\end{corollary}

\vspace{1mm}
For example, let $S$ be a nonempty finite set of primes that does not divide $d$, then it is easy to see that $\mathcal{D}_{0}=\{\phi\}$. Therefore, every nonempty finite set of primes that does not divide $d$ supports all residue patterns for infinitely many primes $p$ of the form $p \equiv f \Mod d$. Here, we mention that the Corollary \ref{allpattern cor} is an analogue of primes in arithmetic progression of (Theorem 2, \cite{MFDR89}) and (Theorem 4.3 of \cite{WRIGHT07}).

\begin{corollary} \label{res cor}
Let $f, d$ be positive integers with $1\leq f \leq d$ and $(f,d)=1$. Necessary and sufficient condition for a nonempty finite set $S \subset \mathbb{Z} \setminus \{0\}$ to be a set of quadratic residue for infinitely many primes of the form $p \equiv f \Mod d$ is given below for each case.
\begin{description}
\item [Case 1] $4 \nmid d$ \\
$\left(\frac{\sqf(T)}{f}\right)=1$, for every $T \in \mathcal{D}_{2}$. 
\item [Case 2] $4 \mid d$ and $8 \nmid d$ \\
$\left(\frac{\sqf(T)}{f}\right)=1$, for every $T \in \mathcal{D}_{1}$.
\item [Case 3] $8 \mid d$ \\
$\left(\frac{\sqf(T)}{f}\right)=1$, for every $T \in \mathcal{D}_{0}$. 
\end{description}
The asymptotic density of primes of the form $p \equiv f \Mod d$ for which all the elements of $S$ are quadratic residues modulo $p$ is 
\begin{equation}\label{density res}
\frac{|\mathcal{D}_{2}|}{\varphi(d)2^{|S|}} \ \ \text{if} \ 4 \nmid d, \ \frac{|\mathcal{D}_{1}|}{\varphi(d)2^{|S|}} \ \ \text{if} \ 4 \mid d, 8 \nmid d \ \text{and} \ \ \frac{|\mathcal{D}_{0}|}{\varphi(d)2^{|S|}} \ \ \text{if} \ 8 \mid d.
\end{equation}
\end{corollary}

\textit{Proof:} Since we expect every element of $S$ to be quadratic residue, the choice of signs is $\theta \equiv 1$. Therefore, the proof of Corollary \ref{res cor} follows from Corollary \ref{inf primes cor}. 
\vspace{1mm}

\begin{corollary} \label{nonres cor}
Let $f, d$ be positive integers with $1\leq f \leq d$ and $(f,d)=1$. Necessary and sufficient condition for a nonempty finite set $S \subset \mathbb{Z} \setminus \{0\}$ to be a set of quadratic non-residue for infinitely many primes of the form $p \equiv f \Mod d$ is given below for each case.
\begin{description}
\item [Case 1] $4 \nmid d$ \\
$\left(\frac{\sqf(T)}{f}\right)=(-1)^{|T|}$, for every $T \in \mathcal{D}_{2}$. 
\item [Case 2] $4 \mid d$ and $8 \nmid d$ \\
$\left(\frac{\sqf(T)}{f}\right)=(-1)^{|T|}$, for every $T \in \mathcal{D}_{1}$.
\item [Case 3] $8 \mid d$ \\
$\left(\frac{\sqf(T)}{f}\right)=(-1)^{|T|}$, for every $T \in \mathcal{D}_{0}$. 
\end{description}
In each case, the asymptotic density of primes for which all the elements of $S$ are quadratic nonresidues modulo $p$ is stated as in \eqref{density res}.
\end{corollary}
\textit{Proof:} Since we expect every element of $S$ to be quadratic nonresidue, the choice of signs is $\theta \equiv -1$. Therefore, the proof of Corollary \ref{nonres cor} follows from Corollary \ref{inf primes cor}. 

\section{A counting problem arising from Theorem \ref{residue thm}}\label{counting cos}
Let $S$ be a finite set of non-zero integers. In this section, we count the number of choice of signs $\theta$ for $S$ such that $S$ has residue pattern $\theta$ modulo $p$ for infinitely many primes $p$ of the form $p \equiv f \Mod d$ (See, Corollary \ref{inf primes cor}). Given $\theta: S \rightarrow \{-1, 1\}$, we define
\begin{equation}\label{Ntheta def}
N_{\theta}=N_{\theta}(S)=\{s \in S: \theta(s)=-1\}.
\end{equation}
Let $\mathbf{F}(S)$ denote the set of all choice of signs $\theta$ on $S$, which is an abelian group under pointwise multiplication for any $\theta_{1}, \theta_{2} \in \mathbf{F}(S)$
$$ \theta_{1}\cdot\theta_{2}(s)=\theta_{1}(s) \cdot \theta_{2}(s), \ \ \ \text{for every} \ s \in S.$$
For a positive integer $d$, $\mathbb{Z}_{d}^{*}=\{1 \leq f < d : (f,d)=1\}$ is a group of units of the ring $\mathbb{Z}_{d}$ of order $\varphi(d)$. Let $\mu^{\theta}_{f}$ be a map defined as in Lemma \ref{homo lem}. We define the following subsets 
\begin{equation}\label{d-subgps defn}
\mathbf{D_{i}}(S)= \left\{(f, \theta) \in \mathbb{Z}_{d}^{*} \times \mathbf{F}(S) : \mu_{f}^{\theta}(T)= \theta(T) \bigg(\frac{\sqf(T)}{f}\bigg)=1, \ \forall \ T \in \mathcal{D}_{i} \right\}
\end{equation}
of the direct product $\mathbb{Z}_{d}^{*} \times \mathbf{F}(S)$ of the groups $\mathbb{Z}_{d}^{*} $ and $\mathbf{F}(S)$, for $i=0,1,2$. It turns out that the subsets $\mathbf{D_{i}}(S)$ are actually subgroups of $\mathbb{Z}_{d}^{*} \times \mathbf{C}(S)$ for $i=0, 1, 2$.

\begin{lemma}\label{sg lemma}
$\mathbf{D_{i}}(S)$ is a subgroup of $\mathbb{Z}_{d}^{*} \times \mathbf{F}(S)$, for $i=0, 1, 2$.
\end{lemma}
\begin{proof}
Here, we prove that $\mathbf{D_{0}}(S)$ is a subgroup of $\mathbb{Z}_{d}^{*} \times \mathbf{F}(S)$. The proof for $\mathbf{D_{1}}(S)$ and  $\mathbf{D_{2}}(S)$ being the subgroups follows similarly. Suppose $(f_{1},\theta_{1}), (f_{2}, \theta_{2}) \in \mathbf{D_{0}}(S)$ then $\mu_{f_{1}}^{\theta_{1}}(T)=1$ and $\mu_{f_{2}}^{\theta_{2}}(T)=1$ holds, for all $T \in \mathcal{D}_{0}$. Let $f \equiv f_{1}f_{2} \Mod d$. Since $\sqf(T) \mid d$ for every $T \in \mathcal{D}_{0}$, we have $f \equiv f_{1}f_{2} \Mod {\sqf(T)}$ for every $T \in \mathcal{D}_{0}$. Therefore, 
\begin{align*}
\mu_{f}^{\theta_{1}\theta_{2}}(T)=\theta_{1}(T) \theta_{2}(T) \bigg(\frac{\sqf(T)}{f}\bigg)=\theta_{1}(T) \bigg(\frac{\sqf(T)}{f_{1}}\bigg)\theta_{2}(T) \bigg(\frac{\sqf(T)}{f_{2}}\bigg)=1.
\end{align*} 
Thus, we obtain $\mu_{f}^{\theta_{1}\theta_{2}} \in \mathbf{D_{0}}(S)$. This completes the proof of Lemma \ref{sg lemma}.
\end{proof}

\begin{thm}\label{counting signs}
$$|\mathbf{D}_{i}(S)|= \frac{2^{|S|} \varphi(d)}{|\mathcal{D}_{i}|}, \ \text{for} \ i=0, 1, 2. $$
\end{thm} 
 
\noindent
We require the following notations from (Section 2, \cite{KBAM21}) for the proof of the Theorem \ref{counting signs}.
\subsection{Notations}
We define a map, $\chi : (\mathbf{P}(S), \triangle) \rightarrow (\mathbb{F}_{2}^{n}, +)$, by
\begin{equation*}
\chi(T)= (\chi_{T}(a_{1}), \dots, \chi_{T}(a_{n})),
\end{equation*}
where for every subset $T \subseteq S$, $\chi_{T}: S \rightarrow \{0,1\}$ denotes the characteristic function of $T$. It is easy to see that $\chi$ is an isomorphism. Also, observe that $\mathbb{F}_{2}^{n}$ is a vector space over the field $\mathbb{F}_{2}$. For any $u=(u_{1},u_{2}, \dots, u_{n}),v=(v_{1},v_{2},\dots,v_{n}) \in \mathbb{F}_{2}^{n}$, $$\langle u, v \rangle =(u_{1}v_{1}+ \dots +u_{n}v_{n}) \Mod 2,$$  defines a symmetric bilinear form on $\mathbb{F}_{2}^{n}$. Thus for any $v \in \mathbb{F}_{2}^{n}$, the map $\phi_{v} : \mathbb{F}_{2}^{n} \rightarrow \mathbb{F}_{2}$ defined by $\phi_{v}(w)=  \langle v, w \rangle$ is a linear functional with the kernel
\begin{equation}\label{max sub}
E(v)=\{ w \in \mathbb{F}_{2}^{n}: \langle v, w \rangle=0 \}.
\end{equation}
We observe that 
\begin{equation}\label{index-2 sbgp}
\chi^{-1}(E(v))= E_{A}(S)=\{T \in \mathbf{P}(S) : |T \cap A| \ \text{is even} \}, \ \text{where} \ \chi(A)=v. 
\end{equation}
Let $D_{i}=\chi(\mathcal{D}_{i})=\{\chi(T): T \in \mathcal{D}_{i}\}$ be the subspace of $\mathbb{F}_{2}^{n}$ corresponding to $\mathcal{D}_{i}$ for $i=0,1,2$.

\vspace{1mm}
On the other hand, for any $f \in \mathbb{Z}_{d}^{*}$, we consider a map $\eta_{f}: S \rightarrow \{-1, 1\}$ by $\eta_{f}(a_{i})=\big(\frac{a_{i}}{f}\big)$, for $i=1, 2, \dots, n$. By the complete multiplicativity of Legendre symbol, $\eta_{f}$ can be defined on $\mathbf{P}(S)$ as follows
$$\eta_{f}(T)= \prod_{a_{i} \in T} \bigg(\frac{a_{i}}{f}\bigg)= \bigg(\frac{\prod_{a_{i} \in T}a_{i}}{f}\bigg)=\bigg(\frac{\sqf(T)}{f}\bigg).$$

\vspace{2mm}
\noindent
\textbf{Proof of Theorem \ref{counting signs}: }
Observe that $\mathbf{F}(S)$ is isomorphic to $\mathbf{P}(S)$ through the map $\theta  \rightarrow N_{\theta}$, and by the relation $N_{\theta_{1}\cdot\theta_{2}}=N_{\theta} \triangle N_{\theta_{2}}$, where $N_{\theta}$ is defined as in \eqref{Ntheta def}. 

 Our aim is to prove $|\mathbf{D_{0}}(S)|=\frac{2^{|S|} \varphi(d)}{ |\mathcal{D}_{0}|}$. The proofs for the other subgroups are similar. Recall from the definition of $\mathbf{D}_{0}(S)$ that an element $(f,\theta) \in \mathbf{D}_{0}(S)$ if $|T \cap N_{\theta}| \equiv |T \cap N_{\eta_{f}}| \Mod 2$, for all $T \in \mathcal{D}_{0}$. From \eqref{max sub} and \eqref{index-2 sbgp}, it is equivalent to say an element $(f,\theta) \in \mathbf{D}_{0}(S)$ if $\langle v, t \rangle = \langle x, t \rangle$, for all $t \in D_{0}$, where $\chi(N_{\eta_{f}})=x$ and $\chi(N_{\theta})=v$.
Therefore, it is enough to prove that, for every $f \in \mathbb{Z}_{d}^{*}$, 
\begin{equation*}
|E_{f}|= |\{v \in \mathbb{F}_{2}^{n}: \langle v, t \rangle= \langle x, t \rangle, \forall \ t \in D_{0} \}|= \frac{2^{|S|}}{|D_{0}|}, \ \text{where} \ \chi(N_{\eta_{f}})=x.
\end{equation*} 
Equivalently,
\begin{equation*}\label{number of lf eq}
|E_{f}|= |\{ v \in \mathbb{F}_{2}^{n}: \langle v+x, t \rangle=0, \forall \ t \in D_{0} \}|= \frac{2^{|S|}}{|D_{0}|}, \ \text{where} \ \chi(N_{\eta_{f}})=x. 
\end{equation*}
Now, following the same arguments as discussed in (Theorem 2, \cite{KBAM21}), we obtain $|E_{f}+x|= \frac{2^{|S|}}{|D_{0}|}$ from which we conclude $|E_{f}|= \frac{2^{|S|}}{|D_{0}|}$ holds for every $f \in \mathbb{Z}_{d}^{*}$.

\section{The degree of \texorpdfstring{$\mathbb{Q}(\sqrt{a_{1}}, \sqrt{a_{2}}, \dots, \sqrt{a_{n}}, \zeta_{d})$} \ \ over \texorpdfstring{$\mathbb{Q}$}.}\label{degree sec}
In this section, we give an exact formula for the degree of $\mathbb{Q}(\sqrt{a_{1}}, \sqrt{a_{2}}, \dots, \sqrt{a_{n}}, \zeta_{d})$ over $\mathbb{Q}$. First, we state the Chebotarev's Density Theorem as follows.

\vspace{2mm}
\noindent
\textbf{Chebotarev's Density Theorem. }
Let $\mathbb{K}/\mathbb{Q}$ be a Galois extension with Galois group $\Gal{\mathbb{K}/\mathbb{Q}}$. Let $C$ be a given conjugacy class in $\Gal{\mathbb{K}/\mathbb{Q}}$. For any rational prime $p$, let $\sigma_{p}$ be the Frobenius element in $\Gal{\mathbb{K}/\mathbb{Q}}$. Then the relative density of the primes $\{p : \sigma_{p} \in C\}$ is $\frac{|C|}{[\mathbb{K}: \mathbb{Q}]}$. 
\begin{lemma}\label{prime fact lem}
Let $m$ be a square-free integer and let $K= \mathbb{Q}(\sqrt{m})$ be a quadratic extension over $\mathbb{Q}$. Let $O_{K}$ be the ring of integers of $K$. Then for any odd prime $p \geq 3$, we have
\begin{enumerate}[label=(\roman*)]
\item $p$ ramifies in $O_{K}$ if and only if $p \mid m$.
\item $p$ splits completely in $O_{K}$ if and only if $\left(\frac{m}{p} \right) = 1$, or $m$ is a square modulo $p$.
\item $p$ is inert in $O_{K}$ if and only if $\left(\frac{m}{p} \right) = -1$, or $m$ is not a square modulo $p$.
\end{enumerate}
\end{lemma}

\begin{thm}\label{degree thm}
Let $S=\{a_1,a_2,...,a_n\}$ be a finite set of non-zero integers. Let $\mathbb{K}=\mathbb{Q}(\sqrt{a_1},\sqrt{a_2},...,\sqrt{a_n},\zeta_d)$ be a multi-quadratic field compositum with cyclotomic extension, where $\zeta_d$ denotes the primitive $d$-th root of unity for $d\geq3$. Then, we have 
\begin{equation*}
[\mathbb{K}: \mathbb{Q}]= \frac{2^{n} \varphi(d)}{|\mathcal{D}_{0}|}.
\end{equation*}
\end{thm}

Note that, in the statement of Theorem \ref{degree thm} we assume $8 \mid d$. The proof of the other two cases follows similarly.

\noindent
\begin{proof}
It is clear that if
$$f(x)= (x^{2}-a_{1}) (x^{2}-a_{2}) \dots (x^{2}-a_{n}) (x^{d}-1) \in \mathbb{Z}[x],$$
then $\mathbb{K}/\mathbb{Q}$ is the splitting field of $f(x)$. Let
$$\mathbf{P}_{1}:=\{p>2: p\equiv 1 \Mod d,\Big(\frac{a_i}{p}\Big)= 1, \ \forall \ 1\leq i \leq n\}.$$
Since $\big(\frac{\sqf(T)}{1}\big)=1$ for all $T \in \mathbf{P}(S)$, from Corollary \ref{res cor}, it follows that the density of $\mathbf{P}_{1}$ is $\frac{|\mathcal{D}_0|}{2^n\varphi(d)}.$ Now, using the Chebotarev's Density Theorem we calculate the relative density of $\mathbf{P}_{1}$. 

\vspace{1mm}
For $p \in \mathbf{P}_{1}$, we want to calculate the Frobenius element $\sigma_p \in \Gal{\mathbb{K}/\mathbb{Q}}$. It is enough to find the action of $\sigma_p$ on $\zeta_d$ and $\sqrt{a_i}$ for each $i$. Since $p \equiv 1 \Mod d$, $\sigma_{p}$ restricted to $\mathbb{Q}(\zeta_{d})$ is the trivial automorphism. Also, since $p \in \mathbf{P}_{1}$, it follows from Lemma \ref{prime fact lem} that $p$ splits completely in $\mathbb{Q}(\sqrt{a_{i}})$ for $1 \leq i \leq n$. Therefore, the Frobenius element $\sigma_p \in \Gal{\mathbb{K}/\mathbb{Q}}$ satisfies $\sigma_{p}(\zeta_{d})= \zeta_{d}$ and $\sigma_{p}(\sqrt{a_{i}})= \sqrt{a_{i}}$ for $1 \leq i \leq n$. Hence, $\sigma_{p}$ is uniquely defined in $\Gal{\mathbb{K}/\mathbb{Q}}$. By the Chebotarev Density Theorem, the relative density of $\mathbf{P}_{1}$ is $$\frac{1}{[\mathbb{K}: \mathbb{Q}]}=\frac{|\mathcal{D}_0|}{2^n\varphi(d)}.$$ This completes the proof of Theorem \ref{degree thm}.
\end{proof}

\section{The explicit Galois group of \texorpdfstring{$\mathbb{Q}(\sqrt{a_{1}}, \sqrt{a_{2}}, \dots, \sqrt{a_{n}}, \zeta_{d})$} \ \ over \texorpdfstring{$\mathbb{Q}$}.}\label{isomorphism sec}
In this section, we discuss the explicit structure of the Galois group of $\mathbb{Q}(\sqrt{a_{1}}, \sqrt{a_{2}}, \dots, \sqrt{a_{n}}, \zeta_{d})$ over $\mathbb{Q}$. Precisely, we prove the following Theorem.
\begin{thm}\label{isomorphism thm}
Let $S=\{a_1,a_2,...,a_n\}$ be a finite set of non-zero integers. Let $\mathbb{K}=\mathbb{Q}(\sqrt{a_1},\sqrt{a_2},...,\sqrt{a_n},\zeta_d)$ denotes a multi-quadratic field compositum with cyclotomic extension, where $\zeta_d$ denotes the primitive $d$-th root of unity for $d\geq3$. Let $\mathbf{D}_0(S)$ be defined as in \eqref{d-subgps defn}. Then there exists an explicit isomorphism between $\mathbf{D}_0(S)$ and $\Gal{\mathbb{K}/\mathbb{Q}}$.
\end{thm}
Note that in the statement of Theorem \ref{isomorphism thm}, we assume $8 \mid d$. The proof for the other two cases follows similarly.

\noindent
\begin{proof}
Since $\mathbb{K}$ is a 2-elementary abelian extension compositum with cyclotomic extension of $\mathbb{Q}$, we have $\Gal{\mathbb{K}/\mathbb{Q}} \cong \mathbb{F}_2^t\times \mathbb{Z}_d^*$ for some $1\leq t\leq n$. On the other hand, since $\mathbf{F}(S) \simeq \mathbb{F}_2^n$, by using Lemma \ref{sg lemma}, we obtain that 
$\mathbf{D}_0(S)\simeq \mathbb{F}_2^t\times \mathbb{Z}_d^*$ for some $1\leq t\leq n$.
Also, it follows from Theorem \ref{counting signs} and Theorem \ref{degree thm} that $|\mathbf{D}_0(S)|=|\Gal{\mathbb{K}/\mathbb{Q}}|=\frac{2^n\varphi(d)}{|\mathcal{D}_0|}.$
Hence, we have $(\Gal{\mathbb{K}/\mathbb{Q}},\circ)\cong(\mathbf{D}_0(S),*)$. In the rest of the proof, we show that there exists an explicit injective homomorphism between $\Gal{\mathbb{K}/\mathbb{Q}}$ and $\mathbf{D}_0(S)$. 

\vspace{1mm}
\noindent
Let $$\mathbf{P}(f,\theta):=\{p>2: p\equiv f \Mod d,\Big(\frac{a_i}{p}\Big)= \theta (a_{i}), \ \forall \ 1\leq i \leq n\}.$$

By using Corollary \ref{inf primes cor}, it follows that $\mathbf{P}(f,\theta)$ contains infinitely many primes, for every $(f, \theta) \in \mathbf{D}_0(S).$ In fact, the relative density of $\mathbf{P}(f,\theta)$ over the set of all primes $\mathrm{P}$ is $\frac{|\mathcal{D}_0|}{2^n\varphi(d)}$. For any $p\in \mathbf{P}(f,\theta)$, we want to calculate the Frobenius element $\sigma_p\in \Gal{\mathbb{K}/\mathbb{Q}}$. Since $p \equiv f \Mod d$, the Frobenius element $\sigma_p(\zeta_d)=\zeta_d^f$.  Therefore, from the definition of $\mathbf{P}(f,\theta)$ and Lemma \ref{prime fact lem} it follows that for any $p \in \mathbf{P}(f, \theta)$ the Frobenius element $\sigma_p \in \Gal{\mathbb{K}/\mathbb{Q}}$ satisfies

\begin{equation*}
  \sigma_p(\sqrt{a_i})=\theta(a_i) \sqrt{a_i}, \ \ \text{for} \ 1\leq i \leq n \ \text {and} \ \sigma_p(\zeta_d)=\zeta_d^f.
\end{equation*}
Hence $\sigma_p$ is uniquely defined in $\Gal{\mathbb{K}/\mathbb{Q}}$. Now, we define a map $\Delta:\mathbf{D}_0(S) \rightarrow \Gal{\mathbb{K}/\mathbb{Q}}$ by 
 $$\Delta((f,\theta))=\sigma_p,  \ \text{for some } p\in \mathbf{P}(f,\theta).$$
 Clearly $\Delta$ is well defined and injective map. We claim that $\Delta$ is a homomorphism. Suppose $\Delta((f_1,\theta_1))=\sigma_{p_{1}}$ and $\Delta((f_2,\theta_2))=\sigma_{p_{2}}$, with $p_1\in \mathbf{P}(f_1,\theta_1)$ and $p_2\in \mathbf{P}(f_2,\theta_2)$, then
 \begin{equation*}
 \Delta((f_1,\theta_1))\circ \Delta((f_2,\theta_2))(\sqrt{a_i})=\sigma_{p_1}(\sigma_{p_2}(\sqrt{a_i}))=(\theta_{1}(a_{i}) \cdot \theta_{2}(a_{i})) \sqrt{a_i}, \  \text{for} \ 1\leq i\leq n
\end{equation*}
and $$\Delta((f_1,\theta_1))\circ \Delta((f_2,\theta_2))(\zeta_d)=\sigma_{p_1}(\sigma_{p_2}(\zeta_d))=\sigma_{p_1}(\zeta_d^{f_2})=\zeta_d^{f_1f_2 \Mod d}=\zeta_d^{f}, $$ where $f\equiv f_1f_2 \Mod d$. On the other hand, let $\Delta((f,\theta_{1}\cdot\theta_{2}))=\sigma_{p}$, for some  $p \in \mathbf{P}(f,\theta_1\cdot\theta_{2})$ and $f\equiv f_1f_2 \Mod d$.  
\begin{equation*}
 \sigma_p(\sqrt{a_i})=\theta_{1}\cdot \theta_{2}(a_{i}) (\sqrt{a_{i}})=(\theta_{1}(a_{i}) \cdot \theta_{2}(a_{i})) \sqrt{a_i}, \ \text{for} \ 1\leq i\leq n \ \ \text{and} \ \sigma_p(\zeta_d)=\zeta_d^{f}.
 \end{equation*}
Thus, we have
$$\Delta((f_1,\theta_1))\circ \Delta((f_2,\theta_2))=\Delta((f_1,\theta_1)\cdot(f_2,\theta_2)).$$
This completes the proof of Theorem \ref{isomorphism thm}.
\end{proof}

\begin{corollary}\label{multi cyclo extn}
Let $S=\{a_1,a_2,...,a_n\}$ be a finite set of non-zero integers and let $3 \leq d_{1} \leq \dots \leq d_{k}$ be integers with $d=$ lcm$(d_{1}, \dots, d_{k})$. Let $\mathbb{L}=\mathbb{Q}(\sqrt{a_1}, \dots ,\sqrt{a_n},\zeta_{d_{1}}, \dots , \zeta_{d_{k}})$ be a multi-quadratic field compositum with several cyclotomic extensions, where $\zeta_{d_{i}}$ denotes the primitive $d_{i}$-th root of unity for $i=1, \dots, k$. Let $\mathbf{D}_{i}(S)$ is defined as in \eqref{d-subgps defn}, for $d=$ lcm$(d_{1}, \dots, d_{k})$, $i=0,1,2$. Then, there is an explicit isomorphism between $\mathbf{D}_0(S)$ and $\Gal{\mathbb{L}/\mathbb{Q}}$.

Here,  we assume $8 \mid d$. The proof for the other two cases follows similarly.
\end{corollary}
\noindent
    The Proof of Corollary \ref{multi cyclo extn} follows similar to Theorem \ref{isomorphism thm} by using (Corollary 4.2.8, \cite{SHWBook}).
    
\section{The distinction between the algebraic cancellations coming from multi quadratic part and cyclotomic part}\label{alg can}  
Let $S=\{a_1,a_2,...,a_n\}$ be a finite set of non-zero integers. In \cite{BLT10}, it has been proved that the exact degree of $\mathbb{Q}(\sqrt{a_1},\sqrt{a_2},...,\sqrt{a_n})$ over $\mathbb{Q}$ is $\frac{2^{n}}{|\mathcal{H}(S)|}$, where
\begin{equation*}
\mathcal{H}=\mathcal{H}(S)=\big\lbrace T \subseteq S : \prod _{s \in T}s= m^{2}, \ \text{for some} \  m \in \mathbb{Z} \big\rbrace.
\end{equation*}

In \cite{KBAM21}, the first two authors pointed out that $(\mathcal{H}(S), \triangle)$ is a subgroup of $(\mathbf{P}(S), \triangle)$. In the following lemma, we will discuss about the cosets of $\mathcal{H}(S)$ in $\mathbf{P}(S)$. 
\begin{lemma}\label{coset lem}
Let $S$ be a finite set of non-zero integers. For any $S_{1} \in \mathbf{P}(S)\setminus \mathcal{H}(S)$, the left coset $S_{1}\cdot \mathcal{H}(S)$ is of the form
\begin{equation*}
\overline{S_{1}}=\left\lbrace T \in \mathbf{P}(S) : \sqf(T)=\sqf(S_{1}) \right\rbrace.
\end{equation*}
In other words, there are exactly $|\mathcal{H}(S)|$ elements of $\mathbf{P}(S)$ such that the product of all the elements of those subsets has the same squarefree part. 
\end{lemma} 
\begin{proof}
Let $S_{1}$ be a nonempty subset of $S$ which is not in $\mathcal{H}(S)$. Consider
\begin{equation*}
S_{1}\cdot\mathcal{H}(S)= \{ S_{1} \triangle T : T \in \mathcal{H}(S) \ \}.
\end{equation*}
We observe that
\begin{equation*}
\sqf\left(\prod_{s \in S_{1} \triangle T}s\right)= \sqf\left(\left(\prod_{s \in S_{1} \triangle T}s\right)\left(\prod_{s \in S_{1} \cap T}s\right)^{2}\right)= \sqf \left(\left(\prod_{s \in S_{1}}s\right)\left(\prod_{s \in T}s\right)\right)=\sqf \left(\prod_{s \in S_{1}}s\right)
\end{equation*}
for all $T \in \mathcal{H}(S)$. Thus, we have $S_{1} \cdot \mathcal{H}(S) \subseteq \overline{S_{1}}$.

\noindent
Suppose $S_{2} \in \overline{S_{1}}$ with $S_{2} \neq S_{1}$. Now, we want show that there exists $T \in \mathcal{H}(S)$ such that $S_{1} \triangle T=S_{2}$. Consider
\begin{equation*}
\text{sqf}\left(\prod_{s \in S_{1} \triangle S_{2}}s\right)=\text{sqf}\left(\left(\prod_{s \in S_{1} \triangle S_{2}}s\right) \left(\prod_{s \in S_{1} \cap S_{2}}s\right)^{2}\right)=\text{sqf}\left(\left(\prod_{s \in S_{1}}s\right) \left( \prod_{s \in S_{2}}s\right)\right)=1,
\end{equation*}
which implies $S_{1}\triangle S_{2} \in \mathcal{H}(S)$. We choose $T=S_{1} \triangle S_{2}$. Thus, we have $S_{1} \triangle T= S_{1} \triangle (S_{1} \triangle S_{2})= S_{2}$. Therefore, $\overline{S_{1}} \subseteq S_{1} \cdot \mathcal{H}(S)$. This completes the proof of Lemma \ref{coset lem}.
\end{proof}

\vspace{1mm}
Consider the quotient group $\overline{\mathbf{P}(S)}=\mathbf{P}(S)/\mathcal{H}(S)$ with a binary operation $\overline{S_{1}}\cdot \overline{S_{2}}=\overline{S_{1} \triangle S_{2}}$ for $\overline{S_{1}}, \overline{S_{2}} \in \overline{\mathbf{P}(S)}$. It follows from Lemma \ref{coset lem} that any two subsets of $S$ has the same squarefree part if and only if they lie in the same coset of the subgroup $\mathcal{H}(S)$.  Thus, $\sqf(\overline{T})$ is well defined and equal to $\sqf(T)$ for every $\overline{T} \in \overline{\mathbf{P}(S)}$. Conventionally, we write the squarefree part of a perfect square is $1$, hence we assume $\sqf(\overline{\phi})=1$. Thus, the quotient groups corresponding to the subgroups $\mathcal{D}_{i}$, for $i=0,1,2,$ in $\overline{\mathbf{P}(S)}$ are given by
\begin{align*}
\overline{\mathcal{D}_{0}}=& \left\{\overline{T} \in \overline{\mathbf{P}(S)}: |\sqf(\overline{T})| \divides d\right\}\nonumber\\
\overline{\mathcal{D}_{1}}=& \left\{\overline{T} \in \overline{\mathbf{P}(S)}: |\sqf(\overline{T})| \divides d \ \text{and} \ \sqf(\overline{T}) \equiv 1 \ \text{or} \ 3 \Mod 4 \right\}\nonumber\\
\overline{\mathcal{D}_{2}}=& \left\{\overline{T} \in \overline{\mathbf{P}(S)}: |\sqf(\overline{T})| \divides d \ \text{and} \ \sqf(\overline{T}) \equiv 1 \Mod 4 \right\}.
\end{align*}
Therefore, by using Theorem \ref{degree thm} the exact degree of $\mathbb{K}$ over $\mathbb{Q}$ is $\frac{2^{n} \varphi(d)}{|\overline{\mathcal{D}_{0}}||\mathcal{H}(S)|},$ where $|\mathcal{H}(S)|$ accounts for the algebraic cancellation that arises from multi quadratic part and $|\overline{\mathcal{D}_{0}}|$ accounts for the algebraic cancellation arising from cyclotomic part.

\bibliographystyle{plain}  

\end{document}